\documentclass[11pt]{article}
\usepackage{amsmath,amsfonts,amsthm,amssymb, mathtools,makecell,multirow}
\usepackage{mathrsfs, graphicx,color,latexsym, tikz, calc, float,hyperref
}
\usepackage{enumerate}

\usepackage{float}

\usetikzlibrary{shadows}
\usetikzlibrary{patterns,arrows,decorations.pathreplacing}
\makeatletter
\newcommand*\bigcdot{\mathpalette\bigcdot@{.5}}
\newcommand*\bigcdot@[2]{\mathbin{\vcenter{\hbox{\scalebox{#2}{$\m@th#1\bullet$}}}}}
\makeatother

\def\l{\langle} \def\r{\rangle}

\newcommand\PSL{\mathrm{PSL}}
\newcommand\Aut{\mathrm{Aut}}

\newcommand\A{\mathrm{A}}

\newtheorem{theorem}{Theorem}[section]

\newtheorem{lemma}{Lemma}[section]

\newtheorem{corollary}{Corollary}[section]

\newtheorem{definition}{Definition}[section]

\allowdisplaybreaks

\title{\bf \Large The Terwilliger algebras of the group association schemes of non-abelian
finite groups admitting an abelian subgroup of index 2
\footnote{  This research was supported by  NSFC (Nos. 12271527, 12071484). E-mail addresses: yjlisy@163.com (J. Yang),
qhguo1001@csu.edu.cn (Q. Guo),  wjliu6210@126.com (W. Liu), fenglh@163.com (L. Feng, corresponding author).}}
\author{
{\small  Jing Yang, \ \ Qinghong Guo,\ \  Weijun Liu,\ \ Lihua Feng$\dagger$}}

\begin{document}
\maketitle
\begin{abstract}
In this paper, we determine the dimension of the Terwilliger algebras of non-abelian
finite groups admitting an abelian subgroup of index 2 by showing that they are triply transitive.
Moreover, we give a complete characterization of the Wedderburn components of the Terwilliger algebras of these groups. \\


\end{abstract}

\section{Introduction}

 Association schemes serve as a foundational concept that integrates various domains, including coding theory, design theory, algebraic graph theory, and finite group theory.

\begin{definition}
Let $X$ be a finite set with cardinality $|X|$, and $R_0,R_1,\ldots,R_\ell$ be a collection of non-empty subsets of $X \times X$.
 Let $N_0, N_1, \ldots, N_\ell$ be matrices in $M_{|X|}(\mathbb{C})$ defined as $(N_i)_{xy}=1$ if $(x, y) \in R_i$, and $0$ otherwise.
 Then $(X,\{N_i\}_{0\le i \le \ell})$ is called an association scheme if
the following conditions are satisfied:
\begin{itemize}
\item[\rm{(i)}] $N_0=I_{|X|}$,
\item[\rm{(ii)}] $N_0+N_1+\cdots+N_\ell=J_{|X|}$,
\item[\rm{(iii)}] $N_i^{\top}=N_{i^{\prime}}$ for some $i^{\prime} \in\{0,1, \ldots, \ell\}$,
\item[\rm{(iv)}] for all $i, j \in$ $\{0,1, \ldots, \ell\}$, we have $N_i N_j=\sum_{k=0}^\ell p_{i j}^k N_k$,
\end{itemize}
where $I_{|X|}$ and $J_{|X|}$ are the identity matrix and the all ones matrix of size $|X|$, respectively, and $N^{\top}$ denotes the transpose of $N$, and $p_{i j}^k=|\{z\in X \mid  (x,z)\in R_i, (z, y)\in R_j\}|$ is a constant for all $i,j,k\in \{0,1,\ldots,\ell\}$ and $(x, y)\in R_k$.
\end{definition}

Terwilliger algebras were first introduced to study $P$- and $Q$-polynomial
association schemes \cite{TT1} and their connection to the orthogonal polynomials
in the Askey scheme. Since their introduction, the Terwilliger algebras of
various families of association schemes have been explored, and several
applications have been developed, one can see \cite{GB,NB,DA,AH,AH2,AH33,HL11,HL,HLA,HHH,KK,AA,YY,XJJ}.

We are now focusing on the Terwilliger algebras that arise from groups.
 Let $X=\{1,2,\ldots,n\}$, and $(X,\{N_i\}_{0\le i \le \ell})$ be an association scheme.
 The algebra $\mathcal{A}$ spanned by $\{N_0,N_1,\ldots,N_\ell\}$ over $\mathbb{C}$ is  a subalgebra of $M_n(\mathbb{C})$ and is
called the \textbf{Bose-Mesner algebra}.
 For $i\in \{0,1,\ldots,\ell\}$ and a fixed element $x\in X$, we define $E^*_{i,x}$ to be the $n\times n$ diagonal matrix whose $(u,u)$ entry is given by $(N_i)_{xu}$. The algebra $\mathcal{A^*}$ spanned by $\{E^*_{0,x},E^*_{1,x},\ldots,E^*_{\ell,x}\}$ over $\mathbb{C}$ is also
 a subalgebra of $M_n(\mathbb{C})$ and is
called the \textbf{dual Bose-Mesner algebra}.
  The Terwilliger algebra of the association scheme $(X,\{N_i\}_{0\le i \le \ell})$ with respect to $x$ is the subalgebra of $M_n(\mathbb{C})$ generated by $\mathcal{A}\cup\mathcal{A^*}$, denoted by $\mathcal{T}(x)$. The Terwilliger algebra of a group $G$, denoted by $\mathcal{T}(G)$, specifically refers to the Terwilliger algebra $\mathcal{T}(e)$, where $e$ is the identity element of $G$. The generators $\{E_{i,e}^*\mid 0\le i\le \ell\}$ of $\mathcal{T}(G)$ are denoted by $\{E_{i}^*\mid 0\le i\le \ell\}$.

 Given a finite group $G$ with identity element $e$ and a subset $S\subseteq G$. The subset $S$ is called a \textbf{conjugacy class} of $G$ if $S=\{gsg^{-1}\mid s\in S,g\in G\}$. Let $R=\{\mathrm{\mathbf{Cl}}_0=\{e\}, \mathrm{\mathbf{Cl}}_1, \ldots, \mathrm{\mathbf{Cl}}_\ell\}$ be the collection of all conjugacy classes of $G$.
 For each $i\in \{0, 1,...,\ell\}$, we define $N_i=\{(x,y)\in G\times G\mid yx^{-1}\in \mathrm{\mathbf{Cl}}_i\}$. It can be verified that
the pair $(G,\{N_i\}_{0\le i \le \ell})$ forms an association scheme, called the \textbf{group association scheme} of $G$.
By the Terwilliger algebra of a group $G$, we mean the Terwilliger algebra of the group association scheme of $G$.
The structure of the Terwilliger algebras of some groups has been studied in \cite{JJ,3,BM,NLB,MR,YZF}, including abelian groups, dihedral groups,
dicyclic groups, $A_5$, $S_5$, $A_6$, $S_6$, $\PSL(2,7)$.
Recently, Maleki \cite{MR} investigated the Terwilliger algebra of $C_n\rtimes C_2$. In \cite{YGZF,YZF}, the authors studied
the Terwilliger algebras of $\mathrm{T}_{n,k}$, $C_n\rtimes C_p$, $C_p\rtimes C_n$, $\mathrm{Dih}(A)$, and $G_2$.
The five families of metacyclic groups as defined in \cite{YGZF,YZF} are as follows:
\begin{enumerate}[(i)]
\item $\mathrm{T}_{n,s}=\left\langle a,b\mid a^{2n}=1,b^2=a^n, bab^{-1}=a^s \right\rangle$,
where $s\not\equiv1~(\bmod~2n)$, ${s}^2\equiv1~(\bmod~2n)$ and $\gcd(2n,s)=1$;
\item $C_n\rtimes C_p=\left\langle a,b\mid a^{n}=b^p=1, bab^{-1}=a^s \right\rangle$,
where $s\not\equiv1~(\bmod~n)$, ${s}^p\equiv1~(\bmod~n)$ and $\gcd(n,s)=1$;
\item $C_p\rtimes C_n=\left\langle a,b\mid a^{p}=b^n=1, bab^{-1}=a^s \right\rangle$,
where $s\not\equiv1~(\bmod~p)$, ${s}^n\equiv1~(\bmod~p)$ and $\gcd(p,s)=1$;
\item $\mathrm{Dih}(A)=\left\langle A, b \mid b^2=1, bab^{-1}=a^{-1}(a \in A)\right\rangle$, where $A$ is an abelian group;
\item $G_2=\l a, b\mid a^n=1, b^2=a^t,bab^{-1}=a^s\r$, where $s^2\equiv1~(\bmod~n)$, $s\not\equiv1~(\bmod~n)$, $\gcd(s,n)=1$, and $t(s-1)\equiv0~(\bmod~n)$.
\end{enumerate}

In this paper, motivated by the foregoing results, we consider the Terwilliger algebras of non-abelian
finite groups admitting an abelian subgroup of index 2. In addition, we determine the Wedderburn decomposition of the Terwilliger algebras of these groups. Our results extend the main results of \cite{JJ,BM,MR,YGZF,YZF}.

\section{Preliminaries}
In this section, we will present some fundamental knowledge in representation theory \cite{Res} and several lemmas which will be used in the sequel.

A (complex) \textbf{representation} of a given finite group $G$ is a group homomorphism $\rho$ : $G \rightarrow \mathrm{GL}(V)$, where $V$ is a finite dimensional complex vector space. The dimension $\operatorname{dim}(\rho)$ of $\rho$ is the dimension of $V$ and the \textbf{character} of $\rho$ is the map $\chi_\rho: G \rightarrow \mathbb{C}$ defined by $\chi_\rho(g)=\operatorname{Tr}(\rho(g))$ for $g \in G$.

A subspace $W$ of $V$ is invariant under $\rho$ if $\rho(g)(W)=W$ for $g \in G$, and $\rho$ is irreducible if $V \neq\{0\}$ and if $\{0\}, V$ are the only invariant subspaces of $V$. Letting $L^2(G)$ be the vector space of all functions $f: G \rightarrow \mathbb{C}$, together with the Hermitian inner product is given by
$$
\l f_1, f_2\r=\frac{1}{|G|} \sum_{g \in G} f_1(g) \overline{f_2(g)}
$$
for $f_1, f_2 \in L^2(G)$, the representation $\rho$ is \textbf{irreducible} if and only if $\l\chi_\rho ,\chi_\rho\r=1$.
Two representations $\rho_1: G \rightarrow \mathrm{GL}\left(V_1\right)$ and $\rho_2: G \rightarrow \mathrm{GL}\left(V_2\right)$ of $G$ are \textbf{equivalent} if there is an isomorphism $T: V_1 \rightarrow V_2$ such that, for every $g \in G$, we have $T \circ \rho_1(g)=$ $\rho_2(g) \circ T$. Classically, $\rho_1$ and $\rho_2$ are equivalent if and only if $\chi_{\rho_1}=\chi_{\rho_2}$. The group $G$ has only finitely many inequivalent irreducible representations, and their number equals the number of conjugacy classes of $G$. Moreover, if $\rho_1, \ldots, \rho_n$ are the inequivalent irreducible representations of $G$, then
$$
|G|=\left(\operatorname{dim}\left(\rho_1\right)\right)^2+\cdots+\left(\operatorname{dim}\left(\rho_n\right)\right)^2 .
$$
In particular, if $G$ is abelian, then every irreducible representation of $G$ has dimension 1.

Let $W=\{N_0,N_1,\ldots,N_\ell\}$  be an association scheme on the finite group $G$.
In \cite{BM}, Bannai and Munemasa considered the $\mathbb{C}$-vector space $\mathcal{T}_0(G)$ given by
$$\mathcal{T}_0(G) = \mathrm{Span}_{\mathbb{C}}\{E_i^*N_jE_k^*\mid 0\le i, j,k \le \ell\}$$
and showed that the dimension of $\mathcal{T}_0(G)$ is given by
$$
\mathrm{dim}_{\mathbb{C}} \mathcal{T}_0(G)=|\{(i,j,k)\mid p_{ij}^k\not=0\}|.
$$

\begin{lemma}\cite{NLB}\label{low}
Let $\mathrm{\mathbf{Cl}}_0, \mathrm{\mathbf{Cl}}_1,\ldots, \mathrm{\mathbf{Cl}}_\ell$ be all conjugacy classes of $G$.
Then $\mathcal{T}_0(G)\subseteq \mathcal{T}(G)$ and $\mathrm{dim}_{\mathbb{C}} \mathcal{T}_0(G)=|\{(i,j,k)\mid \mathrm{\mathbf{Cl}}_k\subseteq \mathrm{\mathbf{Cl}}_i\mathrm{\mathbf{Cl}}_j\}|$.
\end{lemma}

For a finite group $G$ and an element $g\in G$, the \textbf{centralizer} of $g$ in $G$ is denoted
by $\mathrm{C}_G(g)=\{x\in G\mid xg=gx\}$. Suppose $G$ acts on itself by conjugation. Identify
$G$ with the group of permutation matrices $\pi_g$, over all $g\in G$, defined by $(\pi_g)_{x,y} = 1$ if
$gxg^{-1} = y$ and $0$, otherwise. The \textbf{centralizer algebra} $\widetilde{\mathcal{T}}(G)$ consists of all $|G|\times |G|$ matrices
over $\mathbb{C}$ that commute with all $\pi_g$ for all $g\in G$.

\begin{lemma}\cite{BM}\label{up}
$\mathcal{T}(G)\subseteq \widetilde{\mathcal{T}}(G)$ and $\mathrm{dim}_{\mathbb{C}} \widetilde{\mathcal{T}}(G)=\frac{1}{|G|} \sum_{g \in G}\left|\mathrm{C}_G(g)\right|^2$.
\end{lemma}

By Lemmas \ref{low} and \ref{up}, we have $\mathcal{T}_0(G)\subseteq\mathcal{T}(G)\subseteq \widetilde{\mathcal{T}}(G)$.
The group $G$ is called \textbf{triply transitive} if $\mathcal{T}_0(G)=\widetilde{\mathcal{T}}(G)$.

If $G$ is triply transitive, then $\mathcal{T}(G)= \widetilde{\mathcal{T}}(G)$. Since $\mathcal{T}(G)$ is a semisimple algebra,
$\widetilde{\mathcal{T}}(G)$ is a semisimple algebra. Then there exist primitive central idempotents $e_0, e_1,\ldots, e_\ell $
such that
$$
\widetilde{\mathcal{T}}(G)=\bigoplus_{j=0}^{\ell} \widetilde{\mathcal{T}}(G) e_j.
$$

Consider the representation $\Phi: G \rightarrow G L_{|G|}(\mathbb{C})$ such that for any $g \in G$, the matrix $\Phi(g)$ is defined by
$$
\Phi(g)_{u, v}=\left\{\begin{array}{ll}
1, & \text { if } v=gu{g }^{-1}, \\
0, & \text { otherwise, }
\end{array}\right.
$$
for any $u, v \in G$. The representation $\Phi$ is called the \textbf{permutation representation} of the action of $G$ by conjugation on itself. Let $\psi$  be the character corresponding to $\Phi$.
Maleki \cite{MR} showed that the primitive central idempotent $e_j=\frac{\chi_j(1)}{|G|}\sum_{g\in G}\overline{\chi_j(g)}\Phi(g)$
 and $\widetilde{\mathcal{T}}(G) e_j$ is isomorphic to $M_{d_j}(\mathbb{C})$ for each $j\in \{0,1,2,\ldots, \ell\}$.
 Let $g_0, g_1,\ldots, g_\ell$ be a system of representatives of the conjugacy classes
$\mathrm{\mathbf{Cl}}_0, \mathrm{\mathbf{Cl}}_1,\ldots, \mathrm{\mathbf{Cl}}_\ell$, and $\overline{c}$ denotes the complex conjugate of $c\in \mathbb{C}$. Let $\{\chi_0,\chi_1,...,\chi_\ell\}$ be the set of all inequivalent irreducible characters of $G$. Then we have the following lemma
holds.

\begin{lemma}\cite{MR}\label{di}
With the notations described as above, we have
$$
\psi=\sum_{j=0}^\ell d_j\chi_j, \qquad \widetilde{\mathcal{T}}(G)=\bigoplus_{j=0}^{\ell} M_{d_j}(\mathbb{C}),
$$
where $d_j=\sum_{k=0}^\ell\overline{\chi_j(g_k)}$.
\end{lemma}

Let $A$ be a finite abelian group of
order $n\ge3$, and $f\in \Aut(A)$ be of order $2$. Let $y\in A$ be such that $f(y) = y$, and $D_2$ be a non-abelian finite group admitting an abelian subgroup of index $2$. Then $D_2$ admits a presentation
\begin{align}
\begin{aligned}
D_2=\l A, b\mid b^2=y,bab^{-1}=f(a),a\in A\r.
\end{aligned}
\end{align}

For an abelian group $A$, we have $A\cong S_{p_1}\times S_{p_2}\times\cdots \times S_{p_k}$, where $S_{p_1},S_{p_2},\ldots,S_{p_k}$ are
all the Sylow subgroups of $A$. Note that $S_{p_i}\cong C_{p_i^{e_1}}\times C_{p_i^{e_2}}\times \cdots \times C_{p_i^{e_t}}$
with $e_1\le e_2 \le \cdots \le e_t$ and $|S_{p_i}|=p_i^{\sum_{j=1}^te_j}$.
Hence $A$ has the following form
\begin{align}
\begin{aligned}
A&\cong C_{p_1^{e_1}}\times \cdots \times C_{p_\lambda^{e_\lambda}}\times C_{p_{\lambda+1}^{e_{\lambda+1}}}
\times \cdots \times C_{p_{\lambda+\mu}^{e_{\lambda+\mu}}}\\
&\cong \l a_1\r \times\cdots \times \l a_\lambda\r\times
\l a_{\lambda+1}\r\times \cdots \times \l a_{\lambda+\mu}\r.
\end{aligned}
\end{align}
where $p_i$, $i=1,2,\ldots,\lambda+\mu$ are prime and $p_i=2$ if and only if $i=1,\ldots,\lambda$, $\lambda+\mu\not=0$.
Since $f\in \Aut(A)$, we have $f((a_1,a_2,\ldots,a_{\lambda+\mu}))=(a_1^{s_1},a_2^{s_2},\ldots,a_{\lambda+\mu}^{s_{\lambda+\mu}})$, where $s_1,s_2,\ldots,s_{\lambda+\mu}$ are integers.
Let $d_i=\gcd(s_i-1,p_i^{e_i})$ and $n_i=\frac{p_i^{e_i}}{d_i}$ for $i=1,2,\ldots,\lambda+\mu$. 
Let
$$
C(d_i)=\{a_i^{n_ij}\mid 1\le j\le d_i\}, \quad C(n_i)=\{a_i^{d_ij}\mid 1\le j\le n_i\},\quad A^{\prime}=C(d_1)\times C(d_2)\times \cdots \times C(d_{\lambda+\mu}).
$$
Then $B:=\{f(a)a^{-1}\mid a\in A\}=C(n_1)\times C(n_2)\times \cdots \times C(n_{\lambda+\mu})$ and $A/B\cong A^{\prime}$.
\begin{lemma}\label{g2}
Let $a=(a_1^{t_1},\ldots,a_{\lambda+\mu}^{t_{\lambda+\mu}})\in A$ and $D_2$ be the non-abelian group with the presentation given by (1).
Then
\begin{enumerate}
\item[\rm{(i)}]if $a\in A^{\prime}$, then $\mathrm{\mathbf{Cl}}(a)=\{a\}$,
\item[\rm{(ii)}]if $a\notin A^{\prime}$, then $\mathrm{\mathbf{Cl}}(a)=\{a,f(a)\}$,
\item[\rm{(iii)}]$\mathrm{\mathbf{Cl}}(ab)= Bab$.
\end{enumerate}
\end{lemma}
\begin{proof}
For any $a,a^{\prime}\in A$, we have
$${a^{\prime}}^{-1}aa^{\prime}=a, \qquad (a^{\prime}b)^{-1}a(a^{\prime}b)=f(a).$$
Then the conjugate class of $a$ is  $\{a\}$ or $\{a,f(a)\}$. Note that $a=f(a)$ if $a\in A^{\prime}$, and $a\not=f(a)$
if $a\notin A^{\prime}$. Hence the conjugate class of $a$ is $\{a\}$ if $a\in A^{\prime}$, and
the conjugate class of $a$ is $\{a,f(a)\}$ if $a\notin A^{\prime}$.

Similarly, for any $a,a^{\prime}\in A$, we have
$${a^{\prime}}^{-1}aba^{\prime}={f(a^{\prime})}{a^{\prime}}^{-1}ab, \qquad 
(a^{\prime}b)^{-1}ab(a^{\prime}b)=(f({a^{\prime}}^{-1})a^{\prime})(f(a)a^{-1})ab.$$
Hence $\mathrm{\mathbf{Cl}}(ab)=Bab$.
This completes the proof.
\end{proof}

Let $d=\prod_{i=1}^{\lambda+\mu}d_i$. Clearly, there are $d, \frac{n-d}{2}, d$ conjugate classes like
$\{a\}, \{a,f(a)\}$, $\mathbf{Cl}(ab)$, respectively. Denote by
$$\{x_1,x_2,\ldots,x_d\},\qquad\{y_1,y_2,\ldots,y_{\frac{n-d}{2}}\},\qquad\{z_1b,z_2b,\ldots,z_db\}$$
a system of representatives of the conjugacy class like $\{a\}$, $\{a,a^{s}\}$, $\mathrm{\mathbf{Cl}}(ab)$, respectively.
Then $\{x_1,x_2,\ldots,x_d\}=A^{\prime}$ and $
\{z_1,z_2,\ldots,z_d\}=\{a=(a_1^{t_1},\ldots,a_{\lambda+\mu}^{t_{\lambda+\mu}})\in A\mid 1\le t_i\le d_i \text{ for } 1\le i\le \lambda+\mu \}$.
 By Lemma \ref{cc1}, we can describe all conjugacy classes of $D_2$ more clearly.

\begin{lemma}\label{cc1}
All the conjugacy classes of $D_2$ are as follows:
\begin{itemize}
\item[\rm{(i)}] $A_j=\mathrm{\mathbf{Cl}}(x_j)=\{x_j\}$ for $1\le j\le d$,
\item[\rm{(ii)}] $B_j=\mathrm{\mathbf{Cl}}(y_j)=\{y_j,f(y_j)\}$ for $1\le j\le \frac{n-d}{2}$,
\item[\rm{(iii)}] $C_j=\mathrm{\mathbf{Cl}}(z_jb)$ for $1\le j\le d$.
\end{itemize}
Moreover, $|A_j|=1$ and $|B_j|=2$ and $|C_j|=\frac{n}{d}$.
\end{lemma}

In order to get the character table of $D_2$, we first give the characters of $A$.

\begin{lemma}\cite{Res}\label{abe}
Let $a=(a_1^{t_1},\ldots,a_{\lambda+\mu}^{t_{\lambda+\mu}})\in A$. Then we have
$$
\chi_{I}(a)=\prod_{i=1}^{\lambda+\mu}\zeta_{p_i^{e_i}}^{t_ir_i},
$$
where $I=(r_1,\ldots,r_{\lambda+\mu})$, $0\le r_i\le p_i^{e_i}-1$ and $\zeta_{p_i^{e_i}}$ is a primitive $p_i^{e_i}$-th root of unity.
\end{lemma}

The following lemma provides all the irreducible representations of $D_2$.

\begin{lemma}\cite{AB}\label{dex2}
 Let $A$ be an abelian group and $B=\{f(a)a^{-1}\mid a\in A\}$ be a subgroup of $D_2$. Then all irreducible representations of $D_2$ are as follows:
\begin{itemize}
\item[\rm{(i)}]  if $\sigma : A/B \rightarrow \mathbb{C}^*$ is a one-dimensional representation of $A/B$,
then the one-dimensional representations $\rho_1$ and $\rho_2$ of $D_2$ are morphisms of the form
\begin{itemize}
\item[\rm{(i)}]$\rho_1(b)=\sqrt{\sigma(\widetilde {y})}$ and $\rho_1(a)=\sigma(\widetilde {a})$ for $a\in A,$
\item[\rm{(ii)}]$\rho_2(b)=-\sqrt{\sigma(\widetilde {y})}$ and $\rho_2(a)=\sigma(\widetilde {a})$ for $a\in A,$
  \end{itemize}
  where $\widetilde {x}$ denotes the reduction modulo $B$ of $x\in A$.
\item[\rm{(ii)}] if $\sigma : A\rightarrow \mathbb{C}^*$ is a one-dimensional representation
of $A$ with $B\not\subseteq \mathrm{ker}(\sigma)$, then the two-dimensional representation $R_{\sigma}$ of $D_2$ is a morphism of the form
$$
R_\sigma(a)=\begin{pmatrix}\sigma(a)&0\\0&\sigma(f(a))\end{pmatrix}\quad\quad\:R_\sigma(ba)=\begin{pmatrix}0&\sigma(yf(a))\\\sigma(a)
&0\end{pmatrix}\:(a\in A).
$$
\end{itemize}
In particular, $D_2$ has exactly $2[A : B]$ inequivalent one-dimensional representations and $\frac{|A|-[A:B]}{2}$ inequivalent two-dimensional representations, where $[A : B]$ is the index of $B$ in $A$.
\end{lemma}

By Lemmas \ref{abe} and \ref{dex2}, we can obtain the irreducible representations of $D_2$ more clearly.
Let $a=(a_1^{t_1},\ldots,a_{\lambda+\mu}^{t_{\lambda+\mu}})$ and $y=(a_1^{v_1},\ldots,a_{\lambda+\mu}^{v_{\lambda+\mu}})$.
Let $I=(\ell_1,\ldots,\ell_{\lambda+\mu})$ with $0\le \ell_j\le d_j-1$ and $J=(k_1,\ldots,k_{\lambda+\mu})$ with $0\le k_j\le p_j^{e_j}-1$.
\begin{itemize}
\item[(i)] One-dimensional representations: The one-dimensional irreducible representations are one of the following:

the morphism $\phi_{I, 1}: D_2\rightarrow \mathbb{C}^*$ such that $$\phi_{I,1}(a)=\prod_{j=1}^{\lambda+\mu}\zeta_{d_j}^{\ell_jt_j}
,\qquad \phi_{I, 1}(b)=\prod_{j=1}^{\lambda+\mu}\zeta_{d_j}^{\frac{\ell_jv_j}{2}},$$

the morphism $\phi_{I, 2}: D_2\rightarrow \mathbb{C}^*$ such that $$\phi_{I,2}(a)=\prod_{j=1}^{\lambda+\mu}\zeta_{d_j}^{\ell_jt_j}
,\qquad \phi_{I, 2}(b)=-\prod_{j=1}^{\lambda+\mu}\zeta_{d_j}^{\frac{\ell_jv_j}{2}}.$$

\item[(ii)] Two-dimensional representations:  A two-dimensional irreducible representation of $D_2$ is a morphism of the form
$$
\rho_J: D_2 \rightarrow \mathrm{GL}_2(\mathbb{C})
$$
such that
$$
\rho_J(a)=\begin{pmatrix}
\prod_{j=1}^{\lambda+\mu}\zeta_{p_j^{e_j}}^{k_jt_j} & 0 \\
0 & \prod_{j=1}^{\lambda+\mu}\zeta_{p_j^{e_j}}^{k_jt_js_j}
\end{pmatrix} \quad \text { and } \quad
\rho_J(b)=\begin{pmatrix}
0 & \prod_{j=1}^{\lambda+\mu}\zeta_{p_j^{e_j}}^{k_jv_j} \\
1 & 0
\end{pmatrix}.
$$
\end{itemize}

By Lemma \ref{cc1}, we know that all the conjugacy classes of $D_2$ are
$$
\mathrm{\mathbf{Cl}}(x_1),\ldots,\mathrm{\mathbf{Cl}}(x_d),
\mathrm{\mathbf{Cl}}(y_1),\ldots,\mathrm{\mathbf{Cl}}(y_{\frac{n-d}{2}}),
\mathrm{\mathbf{Cl}}(z_1b),\ldots,\mathrm{\mathbf{Cl}}(z_db).
$$
Let
$$
\begin{aligned}
x_j&=(a_1^{x_{j1}},a_2^{x_{j2}},\ldots, a_{\lambda+\mu}^{x_{j(\lambda+\mu)}}) \text{~for~} 1\le j\le d,\\
y_j&=(a_1^{y_{j1}},a_2^{y_{j2}},\ldots, a_{\lambda+\mu}^{y_{j(\lambda+\mu)}}) \text{~for~} 1\le j\le \frac{n-d}{2},\\
z_j&=(a_1^{z_{j1}},a_2^{z_{j2}},\ldots, a_{\lambda+\mu}^{z_{j(\lambda+\mu)}}) \text{~for~} 1\le j\le d,
\end{aligned}
$$
Denote by $\chi_{I,1}$, $\chi_{I,2}$, $\varphi_J$ the characters corresponding to the representations
$\phi_{I, 1}$, $\phi_{I, 2}$, $\rho_J$, respectively. In the following, we give the character table of $D_2$.

\begin{lemma}\label{tab11}
Let $D_2$ be a non-abelian
finite group admitting an abelian subgroup of index 2 with the presentation given by (1). Then the characters of $D_2$ are provided in Table 1.
\end{lemma}
\vspace{-2.0em}
\begin{table}[H]
\centering
\caption{Character table of $D_2$}
\setlength{\tabcolsep}{2.5mm}
\begin{tabular}{c|c|c|c}
\hline \hline$~$ & $A_j=\mathrm{\mathbf{Cl}}(x_j)~(1\le j\le d)$ & $B_j=\mathrm{\mathbf{Cl}}(y_j)~(1\le j\le \frac{n-d}{2})$ & $C_j=\mathrm{\mathbf{Cl}}(z_jb)~(1\le j\le d)$ \\
\hline$\chi_{I,1}$ & $\prod_{i=1}^{\lambda+\mu}\zeta_{d_i}^{\ell_ix_{ji}}$ & $\prod_{i=1}^{\lambda+\mu}\zeta_{d_i}^{\ell_iy_{ji}}$ &
$\prod_{i=1}^{\lambda+\mu}\zeta_{d_i}^{\ell_i(z_{ji}+\frac{v_i}{2})}$ \\
\hline$\chi_{I,2}$ & $\prod_{i=1}^{\lambda+\mu}\zeta_{d_i}^{\ell_ix_{ji}}$ & $\prod_{i=1}^{\lambda+\mu}\zeta_{d_i}^{\ell_iy_{ji}}$ &
$-\prod_{i=1}^{\lambda+\mu}\zeta_{d_i}^{\ell_i(z_{ji}+\frac{v_i}{2})}$ \\
\hline$\varphi_J$ & $\prod_{i=1}^{\lambda+\mu}\zeta_{p_i^{e_i}}^{k_ix_{ji}}+\prod_{i=1}^{\lambda+\mu}\zeta_{p_i^{e_i}}^{k_ix_{ji}s_i}$ & $\prod_{i=1}^{\lambda+\mu}\zeta_{p_i^{e_i}}^{k_iy_{ji}}+\prod_{i=1}^{\lambda+\mu}\zeta_{p_i^{e_i}}^{k_iy_{ji}s_i}$ & $0$  \\
\hline\hline
\end{tabular}
\end{table}

\section{The dimension of $\mathcal{T}(D_2)$}

\begin{theorem}\label{th1}
Let $D_2$ be a non-abelian finite group admitting an abelian subgroup of index 2
with the presentation given by (1).
Then $D_2$ is triply transitive and $\mathrm{dim}_{\mathbb{C}}\mathcal{T} (D_2)=\frac{1}{2}(3nd+n^2+4d^2)$,
where $d=\prod_{i=1}^{\lambda+\mu}d_i$ and $n=|A|$.
\end{theorem}
\begin{proof}
Let $G=D_2$. In order to obtain our result, we need the following two claims.
\medskip

\textbf{Claim 1}: $\mathrm{dim}_{\mathbb{C}} \widetilde{\mathcal{T}}(G)=\frac{1}{2}(3nd+n^2+d^2).$

\medskip

Note that $[G:\mathrm{C}_G(g)]=|\mathrm{\mathbf{Cl}}(g)|$ for $g\in G$.
By Lemmas \ref{up} and \ref{cc1}, we have
$$
\begin{aligned}
\mathrm{dim}_{\mathbb{C}} \widetilde{\mathcal{T}}(G)
&=\frac{1}{2n}\left(\sum_{a \in A^{\prime}}\left|\mathrm{C}_G(a)\right|^2+
\sum_{a \in A \backslash A^{\prime}}\left|\mathrm{C}_G(a)\right|^2+\sum_{a \in A}\left|\mathrm{C}_G(ab)\right|^2\right) \\
& =\frac{1}{2n}\left(\sum_{a \in A^{\prime}}(2n)^2+\sum_{a\in A \backslash A^{\prime}} (n)^2
+\sum_{a \in A}\left(2d\right)^2\right) \\
& =\frac{1}{2n}\left(\sum_{a\in A^{\prime}}4n^2+\sum_{a\in A \backslash A^{\prime}} n^2
+\sum_{a \in A} 4d^2\right) \\
& =\frac{1}{2n}\left(4n^2d+n^2(n-d)+4nd^2\right) \\
&=\frac{1}{2}(3nd+n^2+4d^2).
\end{aligned}
$$

Hence Claim 1 holds.

\medskip

\textbf{Claim 2}: $\mathrm{dim}_{\mathbb{C}}\mathcal{T}_0(G)=\frac{1}{2}(3nd+n^2+4d^2)$.

\medskip

By Lemma \ref{low}, we have $\mathrm{dim}_{\mathbb{C}} \mathcal{T}_0(G)=|\{(U,V,W)\mid W\subseteq U\cdot V\}|$,
where $U, V , W$ are conjugacy classes of $G$. In order to get Claim 2,
we consider the following nine cases.

\medskip

\textbf{Case 1}: If $U=A_{r},V=A_j$ for $1\le r,j\le d$, then $U\cdot V=A_r\cdot A_j=\{x_r\}\cdot \{x_j\}=\{x_rx_j\}=\mathrm{\mathbf{Cl}}(x_rx_j)$.
Hence, there are $w_1=d^2$ contributions in $\mathrm{dim}_{\mathbb{C}}\mathcal{T}_0(G)$.

\medskip

\textbf{Case 2}: If $U=A_{r},V=B_j$ for $1\le r\le d$ and $1\le j\le \frac{n-d}{2}$, then $U\cdot V=A_r\cdot A_j=\{x_r\}\cdot \{y_j,f(y_j)\}=\{x_ry_j,x_rf(y_j)\}=\{x_ry_j,f(x_ry_j)\}=\mathrm{\mathbf{Cl}}(x_ry_j)$.
Hence, there are $w_2=\frac{1}{2}(nd-d^2)$ contributions in $\mathrm{dim}_{\mathbb{C}}\mathcal{T}_0(G)$.

\medskip

\textbf{Case 3}: If $U=A_r,V=C_j$ for $1\le r,j\le d$, then $U\cdot V=A_r\cdot C_j=\{x_r\}\cdot\mathrm{\mathbf{Cl}}(z_jb)=\mathrm{\mathbf{Cl}}(x_rz_jb)$.
Hence, there are $w_3=d^2$ contributions in $\mathrm{dim}_{\mathbb{C}}\mathcal{T}_0(G)$.

\medskip

\textbf{Case 4}: If $U=B_{r},V=A_j$ for $1\le r\le \frac{n-d}{2}$ and $1\le j\le d$, then $U\cdot V=B_r\cdot A_j=\{y_r,f(y_r)\}\cdot \{ x_j\}=\{y_rx_j,f(y_r)x_j\}=\{y_rx_j,f(y_rx_j)\}=\mathrm{\mathbf{Cl}}(y_rx_j)$.
Hence, there are $w_4=\frac{1}{2}(nd-d^2)$ contributions in $\mathrm{dim}_{\mathbb{C}}\mathcal{T}_0(G)$.

\medskip

\textbf{Case 5}: If $U=B_{r},V=B_j$ for $1\le i,j\le \frac{n-d}{2}$, then $U\cdot V=B_r\cdot B_j=
\{y_r,f(y_r)\}\cdot \{ y_j,f(y_j)\}=\{y_ry_j,f(y_ry_j))\}\cup\{y_rf(y_j),f((y_rf(y_j))) \}
=\mathrm{\mathbf{Cl}}(y_ry_j)\cup \mathrm{\mathbf{Cl}}(y_rf(y_j))$.
Note that $y_ry_j\not=y_rf(y_j)$ and $y_ry_j\not=f((y_rf(y_j)))$. Then $\mathrm{\mathbf{Cl}}(y_ry_j)\not=\mathrm{\mathbf{Cl}}(y_rf(y_j))$.
Hence, there are $w_5=\frac{1}{2}(n-d)^2$ contributions in $\mathrm{dim}_{\mathbb{C}}\mathcal{T}_0(G)$.

\medskip

\textbf{Case 6}: If $U=B_r,V=C_j$ for $1\le r \le \frac{n-d}{2}$ and $1\le j\le d$, then $U\cdot V=B_r\cdot C_j
=\{y_r,f(y_r)\}\cdot \mathrm{\mathbf{Cl}}(z_jb)=\mathrm{\mathbf{Cl}}(y_rz_jb)$.
Hence, there are $w_6=\frac{1}{2}(nd-d^2)$ contributions in $\mathrm{dim}_{\mathbb{C}}\mathcal{T}_0(G)$.

\medskip

\textbf{Case 7}: If $U=C_r,V=A_j$ for $1\le r,j\le d$, then $U\cdot V=C_r\cdot A_j=\mathrm{\mathbf{Cl}}(z_rb)\cdot \{x_j\}=
\mathrm{\mathbf{Cl}}(z_rx_jb)$. Hence, there are $w_7=d^2$ contributions in $\mathrm{dim}_{\mathbb{C}}\mathcal{T}_0(G)$.

\medskip

\textbf{Case 8}: If $U=C_r,V=B_j$ for $1\le r\le d$ and $1\le j\le \frac{n-d}{2}$, then
$U\cdot V=C_r\cdot B_j=\mathrm{\mathbf{Cl}}(z_rb)\cdot \{y_j,f(y_j)\}=\mathrm{\mathbf{Cl}}(z_ry_jb)$.
Hence, there are $w_8=\frac{1}{2}(nd-d^2)$ contributions in $\mathrm{dim}_{\mathbb{C}}\mathcal{T}_0(G)$.

\medskip

\textbf{Case 9}: If $U=C_r,V=C_j$ for $1\le r,j\le d$, then $U\cdot V=C_r\cdot C_j
=\mathrm{\mathbf{Cl}}(z_rb)\cdot \mathrm{\mathbf{Cl}}(z_jb)$. Note that for each fixed $r$, we have that
$\mathrm{\mathbf{Cl}}(z_rb)\cdot \mathrm{\mathbf{Cl}}(z_1b), \mathrm{\mathbf{Cl}}(z_rb)\cdot \mathrm{\mathbf{Cl}}(z_2b),
\ldots, \mathrm{\mathbf{Cl}}(z_rb)\cdot \mathrm{\mathbf{Cl}}(z_{2^{\lambda}}b)$ is a partition of $A$.
Hence, there are $w_9=d\cdot\left(d+\frac{1}{2}(n-d)\right)=d^2+\frac{1}{2}\left(d(n-d)\right)$ contributions in $\mathrm{dim}_{\mathbb{C}}\mathcal{T}_0(G)$.

\medskip

From Cases 1--9, we have

$$
\begin{aligned}
\mathrm{dim}_{\mathbb{C}}\mathcal{T}_0(G)&=\sum_{j=1}^9w_j=4\cdot d^2+\frac{1}{2}(n-d)^2
+5\cdot\frac{1}{2}\left(d(n-d)\right)\\
&=\frac{1}{2}(3nd+n^2+4d^2).
\end{aligned}
$$
Hence Claim 2 holds.

\medskip

By Lemmas $\ref{low}$ and $\ref{up}$, we have $\mathrm{dim}_{\mathbb{C}}\mathcal{T}_0(G)\le \mathrm{dim}_{\mathbb{C}}\mathcal{T}(G)\le \mathrm{dim}_{\mathbb{C}}\widetilde{\mathcal{T}}(G)$. By Claims 1 and 2, we have $\mathrm{dim}_{\mathbb{C}}\mathcal{T}_0(G)=\mathrm{dim}_{\mathbb{C}}\widetilde{\mathcal{T}}(G)=\frac{1}{2}(3nd+n^2+4d^2)$.
Therefore, $D_2$ is triply transitive and  $\mathrm{dim}_{\mathbb{C}}\mathcal{T}(D_2)=\frac{1}{2}(3nd+n^2+4d^2)$.
\end{proof}

We now consider the case where $f(a)=a^{-1}$ for $a \in A$, in which case $B=A^2$ and $y^2=1$ (and $A$ has exponent at least 3 ). If $y=1$,
then $D_2$ is the generalized dihedral group $\mathrm{Dih}(A)$, which is given by the following presentation:
$$
\mathrm{Dih}(A)=\left\langle A, b \mid b^2=1, bab^{-1}=a^{-1}(a \in A)\right\rangle .
$$
In this case, we have  $s_i=-1$ for $i=1,2,\ldots,\lambda+\mu$. Therefore, $d=2^{\lambda}$.

If $y \neq 1$ (and so $|A|$ is even), then $D_2$ is the generalized dicyclic group $\mathrm{Dic}(A, b)$, which is given by the following presentation:
$$
\mathrm{Dic}(A, b)=\left\langle A, b \mid b^2=y, bab^{-1}=a^{-1}(a \in A)\right\rangle,
$$
where $|A|=2n$ and $y$ is the unique element of $A$ of order $2$.
In this case, we have $\lambda=1$ and $s_i=-1$ for $i=1,2,\ldots,\lambda+\mu$. Therefore, $d=2$.

\begin{corollary}\cite[Theorem 3.1]{YGZF}
Let $\mathrm{Dih}(A)$ be a generalized dihedral group.
Then $\mathrm{Dih}(A)$ is triply transitive and
$\mathrm{dim}_{\mathbb{C}}\mathcal{T}(\mathrm{Dih}(A))=\frac{n^2}{2}+3n2^{\lambda-1}+2^{2\lambda+1}$.
\end{corollary}

\begin{corollary}
Let $\mathrm{Dic}(A,b)$ be a generalized dicyclic group.
Then $\mathrm{Dic}(A,b)$ is triply transitive and $\mathrm{dim}_{\mathbb{C}}\mathcal{T}(\mathrm{Dic}(A,b))=6n+2n^2+8$.
\end{corollary}

Next, we consider the case  where $A$ is a cyclic group. Then $D_2$ is the non-abelian finite group admitting a cyclic subgroup of index 2, which is given by the following presentation:
$$G_2=\l a, b\mid a^n=1, b^2=a^t,bab^{-1}=a^s\r,$$ where $s^2\equiv1~(\bmod~n)$, $s\not\equiv1~(\bmod~n)$, $\gcd(s,n)=1$, and $t(s-1)\equiv0~(\bmod~n)$. In this case, we have $\lambda=1$ and $d=\prod_{i=1}^{\lambda+\mu}d_i=\gcd(n,s-1)$.

\begin{corollary}\cite[Theorem 4.1]{YGZF}
Let $G_2$ be a non-abelian finite group admitting a cyclic subgroup of index 2.
Then $G_2$ is triply transitive and
$\mathrm{dim}_{\mathbb{C}}\mathcal{T}(G_2)=\frac{1}{2}(3nd+n^2+4d^2)$.
\end{corollary}
\section{Wedderburn decomposition for $\mathcal{T}(D_2)$}

In this section, we find the Wedderburn decomposition for the Terwilliger algebra of the group $D_2$.
Clearly, $D_2$ has $2d$ inequivalent one-dimensional irreducible representations and $\frac{n-d}{2}$ inequivalent two-dimensional irreducible representations. Denote by
$$
\{\chi_{I_{1,1}},\chi_{I_{1,2}},\ldots,\chi_{I_{2d,1}},\chi_{I_{2d,2}}\}
, \qquad \{ \varphi_{J_1},\varphi_{J_2},\ldots,\varphi_{J_{\frac{n-d}{2}}}\}$$
the character set corresponding to all the
inequivalent one-dimensional and two-dimensional irreducible representations of $D_2$, respectively.
Let
$$
I_{r,1}=I_{r,2}=(\ell_{r1},\ell_{r2},\ldots,\ell_{r(\lambda+\mu)}),\qquad J_u=(k_{u1},k_{u2},\ldots,k_{u(\lambda+\mu)})
$$
for $1\le r \le d$ and $1\le u \le \frac{n-d}{2}$. In particular, $J_u\not=(0,0,\ldots,0)$.
In the following, we will determine the Wedderburn decomposition of $D_2$.
\begin{theorem}\label{th2}
Let $D_2$ be a non-abelian finite group admitting an abelian subgroup of index 2
with the presentation given by (1). If $M_m$ is the $m$-dimensional complex matrix algebra, then
$$
\begin{aligned}
\mathcal{T}(D_2)\cong &
\left(\bigoplus_{r=1}^{d}M_{d_{I_{r,1}}}\right)\bigoplus\left(\bigoplus_{r=1}^dM_{d_{I_{r,2}}}\right)\bigoplus \left( \bigoplus_{u=1}^{\frac{n-d}{2}}M_{d_{J_u}}\right),
\end{aligned}
$$
where $d=\prod_{i=1}^{\lambda+\mu}d_i$ and $n=|A|$ and
\begin{eqnarray}
d_{I_{r,1}}=\begin{cases}
\frac{n}{2}+\frac{3d}{2}, & \text{ if } d_i\mid \ell_{ri} \text{ for } 1\le i\le \lambda+\mu,\\
\frac{d}{2}, &\text{ if } d_i\nmid \ell_{ri} \text{ for some } 1\le i\le \lambda+\mu\\
& \text{ and } d_i\mid n_i\ell_{ri} \text{ for } 1\le i\le \lambda+\mu,\\
0, &~otherwise,\\
\end{cases}
\nonumber
\end{eqnarray}

and

\begin{eqnarray}
d_{I_{r,2}}=\begin{cases}
\frac{n}{2}-\frac{d}{2}, &\text{ if } d_i\mid \ell_{ri} \text{ for } 1\le i\le \lambda+\mu,\\
\frac{d}{2}, &\text{ if } d_i\nmid \ell_{ri} \text{ for some } 1\le i\le \lambda+\mu\\
&\text{ and } d_i\mid n_i\ell_{ri} \text{ for } 1\le i\le \lambda+\mu,\\
0,&~otherwise,\\
\end{cases}
\nonumber
\end{eqnarray}

and

\begin{eqnarray}
d_{{J_u}}=\begin{cases}
d,   &\text{ if } d_i\mid k_{ui} \text{ for } 1\le i\le \lambda+\mu,\\
0, &~otherwise.\\
\end{cases}
\nonumber
\end{eqnarray}
\end{theorem}
\begin{proof}
By Lemma \ref{tab11}, we consider the decomposition
$$ \psi=\sum_{r=1}^{d}d_{I_{r,1}}\chi_{I_{r,1}}
+\sum_{r=1}^{d}d_{I_{r,2}}\chi_{I_{r,2}}
+\sum_{u=1}^{\frac{n-d}{2}}d_{J_u}\varphi_{J_u},$$
where $\psi$ is the permutation representation of $D_2$ acting on itself by conjugation.
Let $a=(a_1^{t_1},\ldots,a_{\lambda+\mu}^{t_{\lambda+\mu}})\in A$.
Next, we compute the coefficients $d_{I_r,1},d_{I_r,2},d_{J_u}$.

\medskip

By Lemmas \ref{cc1} and \ref{tab11}, we have
$$
\begin{aligned}
d_{I_{r,1}}
&=\sum_{j=1}^d \overline{\chi_{I_{r, 1}}\left(x_j\right)}+
\sum_{j=1}^{\frac{n-d}{2}} \overline{\chi_{I_{r, 1}}\left(y_j\right)}+
\sum_{j=1}^{d} \overline{\chi_{I_{r, 1}}\left(z_jb\right)} \\
&=\sum_{j=1}^d\prod_{i=1}^{\lambda+\mu}\zeta_{d_i}^{-\ell_{ri}x_{ji}}+
\sum_{j=1}^{\frac{n-d}{2}}\prod_{i=1}^{\lambda+\mu}\zeta_{d_i}^{-\ell_{ri}y_{ji}}+
\sum_{j=1}^d\prod_{i=1}^{\lambda+\mu}\zeta_{d_i}^{-\ell_{ri}(z_{ji}+\frac{v_i}{2})}\\
&=\frac{1}{2}\sum_{j=1}^d\prod_{i=1}^{\lambda+\mu}\zeta_{d_i}^{-\ell_{ri}x_{ji}}+
\frac{1}{2}\sum_{a\in A}\prod_{i=1}^{\lambda+\mu}\zeta_{d_i}^{-\ell_{ri}t_i}+
\sum_{j=1}^d\prod_{i=1}^{\lambda+\mu}\zeta_{d_i}^{-\ell_{ri}(z_{ji}+\frac{v_i}{2})}\\
&=\frac{1}{2}\prod_{i=1}^{\lambda+\mu}\sum_{j=1}^{d_i}\left(\zeta_{d_i}^{-n_i\ell_{ri}}\right)^j+
\frac{1}{2}\prod_{i=1}^{\lambda+\mu}\sum_{j=1}^{p_i^{e_i}}\left(\zeta_{p_i^{e_i}}^{-n_i\ell_{ri}}\right)^j+
\prod_{i=1}^{\lambda+\mu}\left(\zeta_{d_i}^{-\frac{\ell_{ri}v_i}{2}}
\sum_{j=1}^{d_i}\left(\zeta_{d_i}^{-\ell_{ri}}\right)^j\right).
\end{aligned}
$$
\begin{itemize}
\item[(i)] If $ d_i\mid \ell_{ri}$ for $1\le i \le \lambda+\mu$, then $d_{I_{r,1}}=\frac{1}{2}\prod_{i=1}^{\lambda+\mu}d_i+\frac{1}{2}\prod_{i=1}^{\lambda+\mu}p_i^{e_i}
    +\prod_{i=1}^{\lambda+\mu}d_i=\frac{3d}{2}+n$;
\item[(ii)] If $ d_i\nmid \ell_{ri}$ for some $1\le i \le \lambda+\mu$ and $ d_i\mid n_i\ell_i$ for $1\le i \le \lambda+\mu$, then $d_{I_{r,1}}=\frac{1}{2}\prod_{i=1}^{\lambda+\mu}d_i+0+0=\frac{d}{2}$;
\item[(iii)] If $ d_i\nmid \ell_{ri}$ for some $1\le i \le \lambda+\mu$ and $ d_i\nmid n_i\ell_i$ for some $1\le i \le \lambda+\mu$, then $d_{I_{r,1}}=0$.
\end{itemize}

Secondly, we compute $d_{I_{r,2}}$.

$$
\begin{aligned}
d_{I_{r,2}}
&=\sum_{j=1}^d \overline{\chi_{I_{r, 2}}\left(x_j\right)}+
\sum_{j=1}^{\frac{n-d}{2}} \overline{\chi_{I_{r, 2}}\left(y_j\right)}+
\sum_{j=1}^{d} \overline{\chi_{I_{r, 2}}\left(z_jb\right)} \\
&=\sum_{j=1}^d\prod_{i=1}^{\lambda+\mu}\zeta_{d_i}^{-\ell_{ri}x_{ji}}+
\sum_{j=1}^{\frac{n-d}{2}}\prod_{i=1}^{\lambda+\mu}\zeta_{d_i}^{-\ell_{ri}y_{ji}}-
\sum_{j=1}^d\prod_{i=1}^{\lambda+\mu}\zeta_{d_i}^{-\ell_{ri}(z_{ji}+\frac{v_i}{2})}\\
&=\frac{1}{2}\sum_{j=1}^d\prod_{i=1}^{\lambda+\mu}\zeta_{d_i}^{-\ell_{ri}x_{ji}}+
\frac{1}{2}\sum_{a\in A}\prod_{i=1}^{\lambda+\mu}\zeta_{d_i}^{-\ell_{ri}t_i}-
\sum_{j=1}^d\prod_{i=1}^{\lambda+\mu}\zeta_{d_i}^{-\ell_{ri}(z_{ji}+\frac{v_i}{2})}\\
&=\frac{1}{2}\prod_{i=1}^{\lambda+\mu}\sum_{j=1}^{d_i}\left(\zeta_{d_i}^{-n_i\ell_{ri}}\right)^j+
\frac{1}{2}\prod_{i=1}^{\lambda+\mu}\sum_{j=1}^{p_i^{e_i}}\left(\zeta_{p_i^{e_i}}^{-n_i\ell_{ri}}\right)^j-
\prod_{i=1}^{\lambda+\mu}\left(\zeta_{d_i}^{-\frac{\ell_{ri}v_i}{2}}
\sum_{j=1}^{d_i}\left(\zeta_{d_i}^{-\ell_{ri}}\right)^j\right).
\end{aligned}
$$

\begin{itemize}
\item[(i)] If $ d_i\mid \ell_{ri}$ for $1\le i \le \lambda+\mu$, then $d_{I_{r,2}}=\frac{1}{2}\prod_{i=1}^{\lambda+\mu}d_i+\frac{1}{2}\prod_{i=1}^{\lambda+\mu}p_i^{e_i}
    -\prod_{i=1}^{\lambda+\mu}d_i=\frac{n}{2}-\frac{d}{2}$;
\item[(ii)] If $ d_i\nmid \ell_{ri}$ for some $1\le i \le \lambda+\mu$ and $ d_i\mid n_i\ell_i$ for $1\le i \le \lambda+\mu$, then $d_{I_{r,2}}=\frac{1}{2}\prod_{i=1}^{\lambda+\mu}d_i+0-0=\frac{d}{2}$;
\item[(iii)] If $ d_i\nmid \ell_{ri}$ for some $1\le i \le \lambda+\mu$ and $ d_i\nmid n_i\ell_i$ for some $1\le i \le \lambda+\mu$, then $d_{I_{r,2}}=0$.
\end{itemize}

Finally, we compute $d_{J_u}$ for $1\le u\le \frac{n-d}{2}$.
$$
\begin{aligned}
d_{J_u}
&=\sum_{j=1}^d \overline{\varphi_{J_u}\left(x_j\right)}+
\sum_{j=1}^{\frac{n-d}{2}} \overline{\varphi_{J_u}\left(y_j\right)}+
\sum_{j=1}^{d} \overline{\varphi_{J_u}\left(z_jb\right)} \\
&=\sum_{j=1}^d\left(\prod_{i=1}^{\lambda+\mu}\zeta_{p_i^{e_i}}^{k_{ui}x_{ji}}+
\prod_{i=1}^{\lambda+\mu}\zeta_{p_i^{e_i}}^{k_{ui}x_{ji}s_i}\right)+
\sum_{j=1}^{\frac{n-d}{2}}\left(\prod_{i=1}^{\lambda+\mu}\zeta_{p_i^{e_i}}^{k_{ui}y_{ji}}+
\prod_{i=1}^{\lambda+\mu}\zeta_{p_i^{e_i}}^{k_{ui}y_{ji}s_i}\right)\\
&=\frac{1}{2}\sum_{j=1}^d\left(\prod_{i=1}^{\lambda+\mu}\zeta_{p_i^{e_i}}^{k_{ui}x_{ji}}+
\prod_{i=1}^{\lambda+\mu}\zeta_{p_i^{e_i}}^{k_{ui}x_{ji}s_i}\right)+
\frac{1}{2}\sum_{a\in A}\left(\prod_{i=1}^{\lambda+\mu}\zeta_{p_i^{e_i}}^{k_{ui}t_i}+\prod_{i=1}^{\lambda+\mu}
\zeta_{p_i^{e_i}}^{k_{ui}t_is_i}\right)\\
&=\frac{1}{2}\left(\prod_{i=1}^{\lambda+\mu}\sum_{j=1}^{d_i}\left(\zeta_{d_i}^{-k_{ui}}\right)^j+
\prod_{i=1}^{\lambda+\mu}\sum_{j=1}^{d_i}\left(\zeta_{d_i}^{-k_{ui}s_i}\right)^j\right)+
\frac{1}{2}\left(\prod_{i=1}^{\lambda+\mu}\sum_{j=1}^{p_i^{e_i}}\left(\zeta_{p_i^{e_i}}^{-k_{ui}}\right)^j+
\prod_{i=1}^{\lambda+\mu}\sum_{j=1}^{p_i^{e_i}}\left(\zeta_{p_i^{e_i}}^{-k_{ui}s_i}\right)^j\right).
\end{aligned}
$$

Since $f\in \Aut(A)$ be of order $2$, we have $p_i^{e_i}\mid s_i^2-1$ for $1\le i\le \lambda+\mu$.
Note that $(k_{u1},k_{u2},\ldots,k_{u(\lambda+\mu)})\not=(0,0,\ldots,0)$ and $0\le k_{ui}\le p_i^{e_i}-1$.
If $p_i$ is even, then $s_i$ is odd and hence $p_i^{e_i}\mid k_{ui}s_i$ if and only if $k_{ui}=0$.
If $p_i$ is odd, then $p_i^{e_i}\mid s_i+1$ or $p_i^{e_i}\mid s_i-1$. Without loss of generality, we assume that 
$p_i^{e_i}\mid s_i+1$. Then $p_i^{e_i}\mid k_{ui}s_i$ $\Leftrightarrow$ $p_i^{e_i}\mid k_{ui}(s_i+1)-k_{ui}$
$\Leftrightarrow$ $p_i^{e_i}\mid k_{ui}$ $\Leftrightarrow$ $k_{ui}=0$. Hence $p_i^{e_i}\nmid k_{ui}s_i$ for some
$1\le i \le \lambda+\mu$. Consequently, we obtain
$$
\frac{1}{2}\left(\prod_{i=1}^{\lambda+\mu}\sum_{j=1}^{p_i^{e_i}}\left(\zeta_{p_i^{e_i}}^{-k_{ui}}\right)^j+
\prod_{i=1}^{\lambda+\mu}\sum_{j=1}^{p_i^{e_i}}\left(\zeta_{p_i^{e_i}}^{-k_{ui}s_i}\right)^j\right)=0.
$$
Therefore, we have
$$
d_{J_u}=\frac{1}{2}\left(\prod_{i=1}^{\lambda+\mu}\sum_{j=1}^{d_i}\left(\zeta_{d_i}^{-k_{ui}}\right)^j+
\prod_{i=1}^{\lambda+\mu}\sum_{j=1}^{d_i}\left(\zeta_{d_i}^{-k_{ui}s_i}\right)^j\right).
$$
\begin{itemize}
\item[(i)]If $d_i\mid k_{ui}$ for $1\le i\le \lambda+\mu$, then $d_{J_u}=\prod_{i=1}^{\lambda+\mu}d_i=d$;
\item[(ii)]If $d_i\nmid k_{ui}$ for some $1\le i\le \lambda+\mu$, then $d_{J_u}=0$.
\end{itemize}

By Theorem \ref{th1}, $G_2$ is  triply transitive and hence $\mathcal{T}(G_2)=\widetilde{\mathcal{T}}(G_2)$.
By Lemma \ref{di}, the result follows.
\end{proof}

\begin{corollary}\cite[Theorem 3.2]{YGZF}
If $M_m$ is the $m$-dimensional complex matrix algebra, then
$$
\begin{aligned}
\mathcal{T}(\mathrm{Dih}(A))\cong &
\left(\bigoplus_{r=1}^{d}M_{d_{I_{r,1}}}\right)\bigoplus\left(\bigoplus_{r=1}^dM_{d_{I_{r,2}}}\right)\bigoplus \left( \bigoplus_{u=1}^{\frac{n-d}{2}}M_{d_{J_u}}\right),
\end{aligned}
$$
where $d=2^{\lambda}$ and $n=|A|$ and
\begin{eqnarray}
d_{I_{r},1}=\begin{cases}
\frac{n}{2}+\frac{3d}{2}, & \text{ if } 2\mid \ell_{ri} \text{ for } 1\le i\le \lambda,\\
\frac{d}{2}, &\text{ if } 2\nmid \ell_{ri} \text{ for some } 1\le i\le \lambda\\
&\text{ and } 2\mid 2^{e_i-1}\ell_{ri} \text{ for } 1\le i\le \lambda,\\
0, &~otherwise,\\
\end{cases}
\nonumber
\end{eqnarray}

and

\begin{eqnarray}
d_{I_{r},2}=\begin{cases}
\frac{n}{2}-\frac{d}{2}, & \text{ if } 2\mid \ell_{ri} \text{ for } 1\le i\le \lambda,\\
\frac{d}{2}, &\text{ if } 2\nmid \ell_{ri} \text{ for some } 1\le i\le \lambda\\
&\text{ and } 2\mid 2^{e_i-1}\ell_{ri} \text{ for } 1\le i\le \lambda,\\
0, &~otherwise,\\
\end{cases}
\nonumber
\end{eqnarray}

and

\begin{eqnarray}
d_{J_u}=\begin{cases}
d,   &\text{ if } 2\mid k_{ui} \text{ for } 1\le i\le \lambda+\mu,\\
0, &~otherwise.\\
\end{cases}
\nonumber
\end{eqnarray}
\end{corollary}

\begin{corollary}
If $M_m$ is the $m$-dimensional complex matrix algebra, then
$$
\begin{aligned}
\mathcal{T}(\mathrm{Dic}(A,b))\cong &
\left(\bigoplus_{r=1}^2M_{d_{I_{r,1}}}\right)\bigoplus\left(\bigoplus_{r=1}^2M_{d_{I_{r,2}}}\right)\bigoplus \left( \bigoplus_{u=1}^{\frac{n-2}{2}}M_{d_{J_u}}\right),
\end{aligned}
$$
where $2n=|A|$ and
\begin{eqnarray}
d_{I_{r,1}}=\begin{cases}
n+3, & \text{ if } 2\mid \ell_{r1},\\
1, &\text{ if } 2\nmid \ell_{r1} \text{ and } 2\mid 2^{e_1-1}\ell_{r1},\\
0, &~otherwise,\\
\end{cases}
\nonumber
\end{eqnarray}

and

\begin{eqnarray}
d_{I_{r,2}}=\begin{cases}
n-1, & \text{ if } 2\mid \ell_{r1},\\
1, &\text{ if } 2\nmid \ell_{r1} \text{ and } 2\mid 2^{e_1-1}\ell_{r1},\\
0, &~otherwise,\\
\end{cases}
\nonumber
\end{eqnarray}

and

\begin{eqnarray}
d_{J_u}=\begin{cases}
2,   &\text{ if } 2\mid k_{u1},\\
0, &~otherwise.\\
\end{cases}
\nonumber
\end{eqnarray}
\end{corollary}

\begin{corollary}\cite[Theorem 4.2]{YGZF}
If $M_m$ is the $m$-dimensional complex matrix algebra, then
$$
\begin{aligned}
\mathcal{T}(G_2)\cong &
\left(\bigoplus_{r=1}^{d}M_{d_{I_{r,1}}}\right)\bigoplus\left(\bigoplus_{r=1}^dM_{d_{I_{r,2}}}\right)\bigoplus \left( \bigoplus_{u=1}^{\frac{n-d}{2}}M_{d_{J_u}}\right),
\end{aligned}
$$
where $d=\gcd(n,s-1)$ and $n=\frac{|G_2|}{2}$ and
\begin{eqnarray}
d_{I_{r,1}}=\begin{cases}
\frac{n}{2}+\frac{3d}{2}, & \text{ if } d_i\mid \ell_{ri} \text{ for } 1\le i\le 1+\mu,\\
\frac{d}{2}, &\text{ if } d_i\nmid \ell_{ri} \text{ for some } 1\le i\le 1+\mu\\
& \text{ and } d_i\mid n_i\ell_{ri} \text{ for } 1\le i\le 1+\mu,\\
0, &~otherwise,\\
\end{cases}
\nonumber
\end{eqnarray}

and

\begin{eqnarray}
d_{I_{r,2}}=\begin{cases}
\frac{n}{2}-\frac{d}{2}, &\text{ if } d_i\mid \ell_{ri} \text{ for } 1\le i\le 1+\mu,\\
\frac{d}{2}, &\text{ if } d_i\nmid \ell_{ri} \text{ for some } 1\le i\le 1+\mu\\
&\text{ and } d_i\mid n_i\ell_{ri} \text{ for } 1\le i\le 1+\mu,\\
0,&~otherwise,\\
\end{cases}
\nonumber
\end{eqnarray}

and

\begin{eqnarray}
d_{{J_u}}=\begin{cases}
d,   &\text{ if } d_i\mid k_{ui} \text{ for } 1\le i\le 1+\mu,\\
0, &~otherwise.\\
\end{cases}
\nonumber
\end{eqnarray}
\end{corollary}

\section*{Declaration of competing interests}
We declare that we have no conflict of interests to this work.

\section*{Data availability}
No data was used for the research described in the article.

\end{document}